\documentclass[reqno, 12pt]{amsart}
\usepackage{enumerate}
\usepackage{amssymb}
\theoremstyle{plain}
\newtheorem{theorem}{Theorem}[section]
\newtheorem{lemma}[theorem]{Lemma}
\newtheorem{proposition}[theorem]{Proposition}
\newtheorem{corollary}[theorem]{Corollary}
\newtheorem{MainTheorem}{Theorem}
\newtheorem{MainCorollary}[MainTheorem]{Corollary}

\theoremstyle{definition}

\newtheorem{remark}[theorem]{Remark}
\makeatletter
\def\@map#1#2[#3]{\mbox{$#1 \colon #2 \rightarrow #3$}}
\def\map#1#2{\@ifnextchar [{\@map{#1}{#2}}{\@map{#1}{#2}[#2]}}
\makeatother
\newcommand{\evalat}[1]{\bigr\rvert_{#1}}
\newcommand{\set}[2]{\ensuremath{\left\{#1 \,\colon #2\right\}}}
\newcommand{\chull}[1]{\ensuremath{\left\langle #1 \right\rangle}}
\providecommand{\abs}[1]{\lvert#1\rvert}
\newcommand{\eps}{\varepsilon}
\renewcommand{\phi}{\varphi}
\newcommand{\R}{\mathbb{R}}
\newcommand{\N}{\mathbb{N}}
\newcommand{\A}{\mathcal{A}}
\newcommand{\tcal}{\mathcal{T}}
\newcommand{\ccal}{\mathcal{C}}
\newcommand{\NSM}{\mathcal{PM}}
\newcommand{\M}{\mathcal{PSM}}
\newcommand{\NDC}{\mathcal{H}}
\newcommand{\CS}{\mathcal{CS}}
\newcommand{\Markov}{\mathcal{M}}
\DeclareMathOperator{\dist}{dist}
\DeclareMathOperator{\Card}{Card}
\DeclareMathOperator{\Cl}{Cl}
\title[Semiconjugacy to a map of a constant slope]{Semiconjugacy to a map\\ of a constant slope}
\author{Llu\'{\i}s Alsed\`a}
\address{Departament de Matem\`atiques, Edifici Cc, Universitat
Aut\`onoma de Barcelona, 08913 Cerdanyola del Vall\`es, Barcelona,
Spain}
\email{alseda@mat.uab.cat}
\author{Micha\l\ Misiurewicz}
\address{Department of Mathematical Sciences, IUPUI, 402 N. Blackford
    Street, Indianapolis, IN 46202}
\email{mmisiure@math.iupui.edu}
\thanks{The first author has been partially supported by the MEC grant
numbers MTM2008-01486 and MTM2011-26995-C02-01.}
\subjclass[2010]{Primary: 37E05, 37B40}
\keywords{Piecewise monotonotone maps,
semiconjugacy to a map of constant slope,
topological entropy,
interval Markov maps,
measure of maximal entropy}

\date{August 6, 2014}

\begin{document}
\begin{abstract}
It is well known that a continuous piecewise monotone
interval map with positive topological entropy is semiconjugate to a
map of a constant slope and the same entropy, and if it is
additionally transitive then this semiconjugacy is actually a
conjugacy. We generalize this result to
piecewise continuous piecewise monotone interval maps, and
as a consequence, get it also for piecewise monotone graph maps.
We show that assigning to a continuous transitive piecewise
monotone map of positive entropy a map of constant slope conjugate to
it defines an operator, and show that this operator is not continuous.
\end{abstract}
\maketitle

\section{Introduction}\label{sec-int}

It has been known for a long time that a continuous piecewise monotone
interval map with positive topological entropy is semiconjugate to a
map of a constant slope and the same entropy~\cite{MiTh}, and if it is
additionally transitive, then it is conjugate to a map of a constant
slope and the same entropy~\cite{Parry}. The proof by Milnor and
Thurston uses kneading theory and is quite complicated. In the
book~\cite{ALM} we gave a simpler proof, that does not use kneading
theory. Here we present a modification of this proof, that works for
piecewise continuous piecewise monotone interval maps, and
consequently, also for piecewise monotone graph maps.

At the end we show that assigning to a continuous transitive piecewise
monotone map of positive entropy a map of constant slope conjugate to
it defines an operator and show that this operator is not continuous.

To avoid misunderstanding about terminology, let us mention that we
will be speaking about \emph{increasing}
(resp.\ \emph{decreasing, monotone}) and \emph{strictly increasing}
(resp.\ \emph{strictly decreasing, strictly monotone}) maps.
Also, when $f$ is piecewise affine we define the \emph{slope of $f$}
as $\abs{f'(x)}$ for every $x$ where $f$ is differentiable.

We will denote by $\NSM$ the class of interval maps which are
piecewise continuous piecewise monotone (with finitely many pieces),
and by $\M$ the class of all maps from $\NSM$ that are piecewise
strictly monotone. The interval we consider is compact and we will
denote it $I$. The maximal intervals where a map $f$ is simultaneously
monotone and continuous will be called the \emph{laps of $f$}. We
consider them closed and do not worry that if a point is the common
end of two laps then it can have two images. Since we want two laps to
have at most one point in common, if the maximal intervals of
continuity and monotonicity have a larger intersection, we redefine
the laps as slightly shorter intervals. Note that this problem does
not occur if the map is in $\M$. Note also that any lap is a non-degenerate
interval.

For every two points $x,y\in\R$ we denote by $\chull{x,y}$ the
smallest closed interval containing those two points (degenerated to a
point if $x=y$).

We are interested in semiconjugacies $\psi$ of a given map $f \in \M$
of positive entropy to maps $\map{g}{[0,1]}$ with constant slope. This
slope will be $\beta=\exp(h(f))$, where $h(f)$ is the topological
entropy of $f$.
% There is a minor problem if a semiconjugacy contracts a whole lap to a
% point. Then, strictly speaking, we should consider a lap of $g$ which
% consists of one point. Since this is inconsistent with our definition
% of piecewise continuous piecewise monotone maps, we modify slightly
% the definition of semiconjugacy, by dropping the requirement that
% $\psi(f(x))=g(\psi(x))$ if $x$ is in a lap mapped by $\psi$ to a
% point.

We shall restrict our attention to the most natural class of
semiconjugacies, namely increasing ones.
Additionally we will require that if a semiconjugacy $\psi$
maps a lap $J$ to a point, then it also maps the interval
$f(J)$ to a point.
We will call this property \emph{$f$-compatibility}.

We shall denote the class of
increasing continuous maps from $I$ onto $[0,1]$ by $\NDC$
and the class of $f$-compatible elements of $\NDC$ by $\NDC_f$.
Since an inverse image of a point under an increasing continuous map is a
closed interval (perhaps degenerate), a map from $\NDC$ just
identifies to a point each element of some family of disjoint closed
intervals and rescales what is left to get the interval
$[0,1]$. One of the advantages of such an approach is that we have the
following result.

\begin{lemma}\label{E4a.0}
Let $f\in \NSM$, $\psi\in\NDC$ and let $\map{g}{[0,1]}$ be a map such
that $\psi\circ f = g \circ\psi$.
Then for every lap $J$ of $f$ the map $g$ is continuous on $\psi(J)$.
Moreover, if $f\evalat{J}$ is increasing (respectively decreasing)
then $g\evalat{\psi(J)}$ is increasing (respectively decreasing).
In particular, $g \in \NSM$ and the number of laps of $g$ does not
exceed the number of laps of $f$.
\end{lemma}

\begin{proof}
Let $K$ be a compact subset of $[0,1]$. Then, since $\psi$ is a
surjection,
\[
\psi((\psi\circ f\evalat{J})^{-1}(K))=(g\evalat{\psi(J)})^{-1}(K).
\]
Since $\psi$ and $\psi\circ f\evalat{J}$ are continuous, the above set
is compact. This proves that $g$ is continuous on $\psi(J)$.

Assume that $x,y \in \psi(J)$ and $x<y.$ Then, since $\psi$ is
increasing, there are points $x',y'\in J$ such that $\psi(x')=x$,
$\psi(y')=y$, and $x'<y'$.
Assume now that $f$ is increasing on $J$. We get
$g(x) = \psi(f(x')) \le \psi(f(y')) = g(y).$
Thus, $g$ is increasing on $\psi(J)$.
Similarly, if $f$ is decreasing on $J$ then $g$ is decreasing on
$\psi(J)$.
\end{proof}

\section{Dealing with piecewise continuity}\label{sec-pcont}

In most of definitions and theorems in dynamical systems one assumes
that the map which is iterated is continuous. Here we deal with
piecewise continuous maps, so we owe some explanations.

As we already mentioned, we do not really care what the actual values
at the points of discontinuity are. When necessary, we consider
one-sided limits, that exist because of monotonicity. All this can be
formalized (see, e.g.,~\cite{M,Raith}), but this would only make
reading the paper more difficult.

The next problem is the definition of topological entropy of a
discontinuous map. If $f\in\NSM$, the usual solution is to use the
formula
\begin{equation}\label{cn}
h(f)=\lim_{n\to\infty} \frac{1}{n} \log c_n(f),
\end{equation}
where $c_n(f)$ is the number of laps of $f^n$. This is a good
solution, since, as it was shown in~\cite{MZ}, this definition
coincides with the usual Bowen's definitions via $(n,\eps)$-separated
and $(n,\eps)$-spanning sets.

Another problem comes when we try to represent a continuous piecewise
monotone graph map as a piecewise continuous monotone interval map.
A \emph{graph map} is a continuous map form a (finite) graph into itself.
If $J$ is a connected subset of a graph and $f$ is a graph map,
we say that $f$ is \emph{monotone on $J$} if the $f\evalat{J}-$preimages
of points are connected.
A graph map $f$ is \emph{piecewise monotone} if the graph can be divided
into finitely many connected pieces such that $f$ is monotone on each of them.
In this case, we can represent $f$ as a piecewise continuous monotone interval map
by partitioning the graph into pieces, each of them homeomorphic to an interval,
and putting those pieces side by side to get a larger interval.
The cuts are necessary at the vertices of the graph,
and discontinuities may be created at their preimages.
This construction works nice, but we have to know that the
(topological) entropies of the graph map and the interval map are the
same.

Perhaps the simplest way of proving the equality of the entropies in the
above situation is to use the theorem on horseshoes. In both cases,
for continuous graph maps (by~\cite{LM}) and for piecewise continuous
piecewise monotone interval maps (by~\cite{MZ}), there exist sequences
$(k_n)_{n=1}^\infty$ and $(s_n)_{n=1}^\infty$ such that $f^{k_n}$ has
an $s_n$-horseshoe and
\[
\lim_{n\to\infty}\frac{1}{k_n}\log s_n
\]
is the entropy of the corresponding map. Moreover, this limit cannot
be larger than the entropy. The definitions of horseshoes are
consistent with the operator of cutting and putting pieces side by
side. Therefore, the entropies of both maps are equal.

\section{From $\NSM$ to $\M$}\label{sec-nsmm}

It is easier to work with maps that are piecewise strictly monotone,
rather than just piecewise monotone. Fortunately there exists a simple
way to reduce the semiconjugacy problem to the maps from $\M$.

Let $\map{f}{I}$ be a map from $\NSM$. We divide $I$ into the
laps of $f$. As we mentioned earlier, we modify the definition of a
lap in such a way that two different laps have at most one point in
common.

Now we apply the usual coding procedure. If there are $s$ laps, we
consider first the full one-sided $s$-shift $(\Sigma,\sigma)$, and
then restrict it to the set $\Sigma_f$ of those sequences
$(\eps_n)_{n=0}^\infty \in\Sigma$ for which there exists a point $x\in
I$ such that $f^n(x)$ is in the $\eps_n$-th lap for every $n$.
By~\eqref{cn}, the topological entropy of this subshift is equal to
$h(f)$. Clearly, the set
$\Sigma_f$ is $\sigma$-invariant. It is also
closed because the laps are compact.

Let us introduce the following relation on $I$. Given two points $x,y
\in I$ we write $x \sim y$ if they have the same codes. Clearly, $\sim$
is an equivalence relation and its equivalence classes are intervals
(perhaps degenerate) because on each lap $f$ is continuous and
monotone.
Moreover, those intervals are closed because laps are closed. Now we
can consider the quotient space $I/\kern-5pt\sim$ which is a new interval $\widehat{I}$.
The projection $\map{\psi}{I}[\widehat{I}]$ is in fact a continuous increasing map and it is
$f$-compatible.
If $x$ and $y$ are equivalent, then $f(x)$ and $f(y)$ are also
equivalent. Therefore $f$ factors to some map $\map{\widehat{f}}{\widehat{I}}$.
By Lemma~\ref{E4a.0}, we have $\widehat{f}\in\NSM$.
Moreover, $\widehat{f}\in\M$. Indeed, if $\widehat{f}$ is constant on some interval
then, by the definition of $\widehat{f},$
the preimage of this interval under $\psi$ consists of
equivalent points, a contradiction.

We claim that the topological entropies of $f$ and $\widehat{f}$ are equal.
Indeed, the corresponding symbolic systems differ by at most countable
number of points (some points could disappear because the whole lap of
some iterate of $f$ could disappear). However, this does not affect
nonatomic invariant measures, so by the variational principle the
topological entropies of both systems are the same (see, e.g. \cite{Wal}).

The discussion from this section can be summarized as a lemma.

\begin{lemma}\label{strict}
Any map from $\NSM$ of positive entropy is semiconjugate to a map from
$\M$ with the same entropy via a map from $\NDC_f$.
\end{lemma}

Note that the assumption about positive entropy is necessary because
otherwise it could happen that the interval $\widehat{I}$ is degenerate to a
point. This is for instance the case if $f$ is the identity map.

Note that if we construct a semiconjugacy in two steps, first
semiconjugating a map $f\in\NSM$ to $h\in\M$, and then $h$ to a
map of constant slope, we have to know that the resulting
semiconjugacy is $f$-compatible. However, it is easy to see that if
the first semiconjugacy is $f$-compatible and the second one is
$h$-compatible, then their composition is $f$-compatible.

\section{Reduction to semiconjugacies}\label{reduction}

When we want to prove that a map from $\M$ of positive entropy is
semiconjugate to a map with constant slope, it is easier to look not
for this map, but for a semiconjugacy.

Denote the class of all piecewise continuous piecewise monotone maps
from $[0,1]$ into itself with constant slope $\beta$ by $\CS_{\beta}$.
Observe that they are piecewise strictly monotone.

\begin{remark}\label{entconstslope}
By \cite{MS} we have $h(g) = \max\{\log\beta, 0\}$ for very $g
\in \CS_{\beta}.$
\end{remark}

Suppose that $\psi\in \NDC$ semiconjugates $f\in \M$ with
$g \in \CS_{\beta}$ for some $\beta$.
In view of Lemma~\ref{E4a.0}, we have then
\begin{equation}\label{eq4a.1}
\abs{\psi(f(y)) - \psi(f(x))} = \beta\abs{\psi(y) - \psi(x)}
\end{equation}
for all $x,y \in I$ such that $f$ is continuous and monotone on
$\chull{x,y}$. We shall show that the converse is also true.

\begin{lemma}\label{E4a.2}
If $f\in \M,$ $\psi\in \NDC,$ $\beta >1$ and \eqref{eq4a.1} holds for
all $x,y\in I$ such that $f$ is continuous and monotone on
$\chull{x,y},$ then there exists (a unique) $g\in \CS_{\beta}$
such that $\psi$ semiconjugates $f$ with $g$.
\end{lemma}

\begin{proof}
Let us define $g$. We do it separately for every lap $J$ of $f$. If
$\psi$ is constant on $J$, then we ignore $J$.
Assume that $\psi(J)$ is a nondegenerate interval.
If $z\in\psi(J)$ then $(\psi\evalat{J})^{-1}(z)$ is a closed interval
(perhaps degenerated to a point). By~\eqref{eq4a.1} $\psi$ is constant
on $f((\psi\evalat{J})^{-1}(z))$.
This means that
$f((\psi\evalat{J})^{-1}(z)) \subset \psi^{-1}(w)$
for some $w \in [0,1]$.
Then we set $g(z) = w$.
For every $v \in(\psi\evalat{J})^{-1}(z)$ we have
$\psi(f(v)) = w = g(\psi(v))$.
This shows that with our definition of $g$ we get
$\psi\circ f = g\circ\psi$ on $J$.
This is true for every lap $J$ of $f$ which is not mapped by $\psi$ to
a point, so $\psi$ semiconjugates $f$ with $g$.

Clearly, if we want to have $\psi(f(v)) = g(\psi(v))$ then with the
construction described above, there is only one choice for $g(z)$
(except for the points of discontinuity, where we agreed that the map
can have two values). This proves the uniqueness of $g$.

It remains to prove that $g\in \CS_{\beta}$. From Lemma~\ref{E4a.0}
and~\eqref{eq4a.1} it follows that the map $g$ belongs to $\NSM$ and
if $f$ is continuous and monotone on $\chull{x,y}$ then
\[
\abs{g(\psi(y)) - g(\psi(x))} =
  \abs{\psi(f(y)) - \psi(f(x))} =
  \beta\abs{\psi(y) - \psi(x)}.
\]
This shows that $g$ has constant slope $\beta$ on each nondegenerate
image under $\psi$ of a lap of $f$. Since $I$ can be divided into
finitely many laps of $f$, we get $g\in\CS_{\beta}$.
\end{proof}

\section{Markov maps}\label{sec-markov}

Now we consider a special class of maps. We call $f\in \NSM$ a
\emph{Markov map\/} if there exists a finite $f$-invariant set $P$
containing the endpoints of $I$, such that on each (closed) interval
of the partition $\A$ of $I$ by the points of $P$ the map $f$ is either
strictly monotone and continuous or constant. We call $\A$ a
\emph{Markov partition} for $f$. We denote the class of all Markov
maps $\map{f}{I}$ by $\Markov$.

\begin{lemma}\label{E4a.3}
If $f\in \Markov$ and $h(f) = \log\beta > 0$ then $f$ is semiconjugate
to some map $g\in \CS_{\beta}$ via an increasing $f$-compatible map.
\end{lemma}

\begin{proof}
By Lemma~\ref{E4a.2}, it is enough to prove that there exists
$\psi\in\NDC_{f}$ such that \eqref{eq4a.1} holds for all $x,y \in I$ such
that $f$ is continuous and monotone on $\chull{x,y}$.

Let $\A$ be the Markov partition for $f$ by the points of $P$. Let
$\A_n$ be the common refinement of $f^{-i}(\A)$, $i=0,1,\dots,n,$
that is
\[
\A_n = \set{A_0 \cap A_1 \cap \dots \cap A_n}{
    A_0 \in \A,
    A_1 \in f^{-1}(\A),
    \dots
    A_n \in f^{-n}(\A)
}.
\]
It is a Markov partition for $f^i$, $i=1,\dots,n+1$.
Let $P_n$ be the set of all endpoints of elements of $\A_n.$
%Observe that $\A_0 = \A$ and $P_0 = P.$

If $A\in\A_n$ then $f(A)$ is either a singleton or an element of
$\A_{n-1}$. Therefore,
\begin{equation}\label{eq4a.2}
f(P_{n+1}) \subset P_n\subset P_{n+1}
\end{equation}
for every $n \ge 0$ and, by induction, $f^n(A)$ is a singleton or an
element of $\A_0=\A$.

Let $M=(m_{_{AB}})_{A,B\in\A}$ be the matrix given by $m_{_{AB}}=1$ if
$f(A)\supset B$ and $m_{_{AB}}=0$ otherwise.
The sum of all entries of $M^n$ is equal to the cardinality of $\A_n$.
Therefore, by Theorem~2 of~\cite{MZ}, the spectral radius of $M$ is
$\beta$ (since the matrix $M$ is nonnegative, its spectral radius is equal to
the limit of the $n$-th root of the sum of all entries of $M^n$).
By the Perron-Frobenius Theorem, $M$ has a right eigenvector
$v=(v_{_A})_{A\in\A}$ associated to the eigenvalue $\beta$, with
non-negative coordinates. We can normalize $v$ in such a way that
\begin{equation}\label{eq4a.3}
  \sum_{A\in\A} v_{_A} = 1.
\end{equation}
For every $A\in\A$ we have
\begin{equation}\label{eq4a.4}
   \sum_{B\in\A} m_{_{AB}} v_{_B} = \beta v_{_A}.
\end{equation}

Now we define $\psi$ on $Q = \bigcup_{n=0}^\infty P_n$. For a set $A$
and a point $x$ we shall write $A\le x$ if $y \le x$ for all $y\in A$.
For each $x\in P_n$ we set
\[
\psi(x) = \beta^{-n} \sum_{A\in\A_n,\, A\le x} v_{_{f^n(A)}}.
\]
In the above formula we omit those $A\in\A_n$ for which $f^n(A)$ is a
singleton. Since $P_n\subset P_{n+1}$, in order to be sure that $\psi$
is well defined, we have to check that if $x\in P_n$ then
\begin{equation}\label{eq4a.5}
\beta^{-n} \sum_{A\in\A_n,\, A\le x} v_{_{f^n(A)}} =
   \beta^{-n-1} \sum_{B\in\A_{n+1},\, B\le x} v_{_{f^{n+1}(B)}}.
\end{equation}
The partition $\A_{n+1}$ is finer than $\A_n$, so~\eqref{eq4a.5} will
follow if we can prove that for each $A\in\A_n$,
\[
v_{_{f^n(A)}} =
   \beta^{-1} \sum_{B\in\A_{n+1},\, B\subset A} v_{_{f^{n+1}(B)}}.
\]
Since $f^n$ maps each $B$ as above in a continuous and monotone way
onto an element of $\A_1$, it is enough to show that for each
$C\in\A$,
\begin{equation}\label{eq4a.6}
v_{_C} = \beta^{-1} \sum_{D\in\A_1,\, D\subset C} v_{_{f(D)}}.
\end{equation}
However, the sets $f(D)$, where $D\in\A_1$ and $D\subset C$, are
exactly those elements $E$ of $\A$ for which $m_{_{CE}} = 1$.
Hence, we can rewrite \eqref{eq4a.6} as
\[
v_{_C} = \beta^{-1} \sum_{E\in\A} m_{_{CE}} v_{_E},
\]
which holds by \eqref{eq4a.4}. This completes the proof that $\psi$ is
well defined on $Q$.

By the definition, $\psi$ is increasing on $Q.$
On the other hand, the endpoints of $I$ belong to $P = P_0$ and
we have $\A_0 = \A$.
Thus, the left endpoint of $I$ is mapped by $\psi$ to 0 and, by
\eqref{eq4a.3}, the right endpoint of $I$ is mapped by $\psi$ to 1.
Notice that
\[
\sup_{\set{x,y}{\chull{x,y}\in\A_n}}
  \abs{\psi(y) - \psi(x)} \le \beta^{-n},
\]
since by \eqref{eq4a.3} all $v_{_A}$ are bounded by 1. From this it
follows that if $x\in Q$ then
$\psi(x) = \sup\set{\psi(y)}{y < x ,\ y\in Q}$
(unless $x$ is the left endpoint of $I$) and
$\psi(x) = \inf\set{\psi(y)}{y > x,\ y\in Q}$
(unless $x$ is the right endpoint of $I$), and if $x\in I\setminus Q$
then
\[
\sup\set{\psi(y)}{y<x,\ y \in Q} = \inf\set{\psi(y)}{y>x,\ y \in Q}.
\]
Consequently, $\psi$ can be extended to a map from $\NDC$, constant on
the components of $I\setminus Q$.

Now we are ready to show that \eqref{eq4a.1} holds. Let $n\ge 1$ and
let $\chull{x,y}\in \A_n$. Then $f(\chull{x,y}) \in \A_{n-1}$ and by
the definition of $\psi$, \eqref{eq4a.1} holds for
such $x,y$. By taking sums and limits, we obtain \eqref{eq4a.1} for
every $x,y\in \Cl(Q)$ such that $f$ is monotone on $\chull{x,y}$.

Let $z<w$ and $Q\cap (z,w) =\emptyset$. Then for every $n>0$ the
interval $[z,w]$ is contained in some element of $\A_n$, and thus
$f([z,w])$ is contained in some element of $\A_{n-1}$. Therefore
$Q\cap f((z,w)) = \emptyset$.

We already know that formula \eqref{eq4a.1} holds for every
$x,y\in\Cl(Q)$ with $f$ continuous and monotone on $\chull{x,y}$.
Moreover, if $z<w$ and $Q\cap (z,w) = \emptyset$ then $Q\cap f((z,w))
= \emptyset$ and then $\psi(z) = \psi(w)$ and $\psi(f(z)) =
\psi(f(w))$. This proves that \eqref{eq4a.1} holds for all $x,y\in I$
with $f$ continuous and monotone on $\chull{x,y}.$
Moreover, from \eqref{eq4a.1} it follows that $\psi$ is $f$-compatible.
This completes the proof.
\end{proof}

\section{Main theorem}\label{sec-main}

To be able to use Lemma~\ref{E4a.3}, we have to prove the density of
Markov maps in $\M$. The space $\M\cup\Markov$ with the topology of
uniform convergence is a metric space with the natural metric
\begin{equation}\label{dist}
\dist(f,g) = \sup_{x\in I} \abs{f(x) - g(x)}.
\end{equation}

\begin{lemma}\label{E4a.4}
Let $f\in \M$ and $n\in\N$. Then there exists $g\in \Markov$ such that
$\dist(f,g)<1/n$ and for every $k\le n$ the map $g^k$ is continuous
and monotone on every lap of $f^k$.
\end{lemma}

\begin{proof}
Let $Q_k$ be the set of all endpoints of all laps of $f^k$. The set
$Q_k$ is finite and the endpoints of $I$ belong to $Q_k$.

Since $f$ is continuous on each lap of $f$ and each lap of $f$ is
compact, $f$ is uniformly continuous on each lap. Therefore there
exists $\delta$ such that $0 < \delta < \tfrac{1}{2n}$ and if $x,y$
belong to the same lap of $f$ and $\abs{x-y} < \delta$ then
$\abs{f(x) - f(y)} < \tfrac{1}{2n}$.
Take a finite set $P$ containing
$\bigcup_{k=0}^n \bigcup_{i=0}^k f^i(Q_k)$
such that the distance between any two adjacent points of $P$ is
smaller than $\delta$.
For each $x\in P$ set
$g(x) := \max \set{y\in P}{y \le f(x)}$
and extend $g$ affinely between the points of $P$. At the points of
discontinuity of $f$ (they belong to $P$) we take one-sided limits, so
effectively we work separately on each maximal interval of continuity
of $f$. Clearly, $g \in \Markov$. Moreover, $g$ defined in this way is
automatically continuous on each lap of $f$.

Suppose that $x\le y \le z,$ $x,z\in P$ and there are no points of $P$
between $x$ and $z$. We have
$\abs{x-y}\le \abs{x-z} < \delta$,
so
$\abs{f(x) - f(y)} < \tfrac{1}{2n}$.
Moreover, by the definition of $g$ we have
$\abs{f(x) - g(x)} < \delta < \tfrac{1}{2n}$,
and hence
$\abs{f(y) - g(x)} < 1/n$.
Analogously, $\abs{f(y) - g(z)} < 1/n$.
Since $g$ is affine on $[x,z]$, we get
$\abs{f(y) - g(y)} < 1/n$.
This proves that $\dist(f,g) < 1/n$.

Let $k \le n$ and let $J = [a,b]$ be a lap of $f^k$. Then $a,b\in Q_k$
and hence $f^i(a), f^i(b)\in P$ for $i=0,1,\dots,k$.
{}From the definition of $g$, by induction we get
$g^i(a) = f^i(a)$ and $g^i(b) = f^i(b)$ for $i=0,1,\dots,k$.
Look at $g$ on $L_i = \chull{g^i(a),g^i(b)},$ $0 \le i < k$.
Since $L_i = \chull{f^i(a),f^i(b)}$, the map $f$ is monotone on $L_i$.
Therefore $f$ is monotone on $L_i\cap P$. Notice that if $x,y\in P$
and $f(x) \le f(y)$ (respectively $f(x)\ge f(y)$) then also $g(x)\le
g(y)$ (respectively $g(x) \ge g(y)$). Hence, $g$ is monotone on
$L_i\cap P$. The set $L_i\cap P$ contains the endpoints of $L_i$ and
hence $g$ is monotone on $L_i = g^i(J)$.
\end{proof}

\begin{MainTheorem}\label{E4a.5}
If $f\in \NSM$ and $h(f) = \log \beta > 0$ then $f$ is semiconjugate
to some map $g \in \CS_{\beta}$ via a map from $\NDC_f$.
\end{MainTheorem}

\begin{proof}
By Lemma~\ref{strict}, $f$ is semiconjugate to some map $\widehat{f}\in\M$
with the same entropy via a map from $\NDC_f$. Therefore we may assume
for the rest of the proof that $f\in\M$.

By Lemma~\ref{E4a.4}, for each $n$ there exists $f_n\in \Markov$ such
that $\dist(f,f_n)<1/n$ and $f_n^k$ is continuous and monotone on
every lap of $f^k$ for every $k\le n$.
By Theorem~3 of~\cite{MZ}, we have $\liminf_{n\to\infty} h(f_n) \ge
h(f)$ (note that $\lim_{n\to\infty} f_n=f$ in the topology considered
in ~\cite{MZ}).
Therefore there exists a subsequence $(f_{n_i})_{i=1}^\infty$ of the
sequence $(f_n)_{n=1}^\infty$ such that for every $i$ we have
$h(f_{n_i}) = \log \beta_i > 0$,
the limit
$\gamma = \lim_{i\to\infty} \beta_i$ exists and $\log\gamma\ge h(f)$.
The number $\gamma$ is finite since
\[
\beta_i \le \Card(c_1(f_{n_i})) \le \Card(c_1(f))
\]
for all $i$ (recall that $c_1(f)$ denotes the number of laps of $f$).
Hence, by Lemma~\ref{E4a.3}, $f_{n_i}$ is semiconjugate to some
$g_i\in \CS_{\beta_i}$ via some map $\psi_i\in \NDC_{f_{n_i}}$.

Let $\eps > 0$.
Since $\lim_{i\to\infty} \log\beta_i \ge h(f) > 0$
we may assume that
$\alpha := \inf_i \beta_i > 1$.
Then there exists $k$ such that $2\alpha^{-k} < \eps$.
We can find $\delta' > 0$ such that if $\abs{x - y} < \delta'$ then
$\chull{x,y}$ is contained either in a lap of $f^k$ or in a union of
two laps of $f^k$.
In both cases there is $z\in \chull{x,y}$ such that the intervals
$\chull{x,z}$ and $\chull{z,y}$ are contained in laps of $f^k$.
There exists $i(\eps)$ such that if $i \ge i(\eps)$ then $n_i
\ge k$. Hence, for all $i\ge i(\eps)$ the maps $f_{n_i}^k$ are
continuous and monotone on these intervals.
On the other hand, since  $f_{n_i}$ is semiconjugate to a map
in $\CS_{\beta_i},$ it follows that $f_{n_i}^k$ is semiconjugate to a
map in $\CS_{\beta^k_i}.$ Thus, by~\eqref{eq4a.1},
\[
1 \ge
  \abs{\psi_i(f_{n_i}^k(x)) - \psi_i(f_{n_i}^k(z))} =
  \beta_i^k \abs{\psi_i(x) - \psi_i(z)} \ge
  \alpha^k \abs{\psi_i(x) - \psi_i(z)}.
\]
Therefore, $\abs{\psi_i(x) - \psi_i(z)} \le \alpha^{-k}$ and,
analogously, $\abs{\psi_i(z) - \psi_i(y)} \le \alpha^{-k}$.
Hence,
$\abs{\psi_i(x) - \psi_i(y)} \le 2\alpha^{-k} < \eps$.
In such a way we see that
for every $\eps > 0$ there exists $i(\eps)$ and $\delta' > 0$
such that if $i\ge i(\eps)$ and $\abs{x-y} < \delta'$
then
$\abs{\psi_i(x) - \psi_i(y)} < \eps$.
Since the maps $\psi_i$ for $i < i(\eps)$ are uniformly continuous,
we obtain that for every $\eps > 0$ there exists $\delta > 0$ such
that if $\abs{x-y} < \delta$ then
$\abs{\psi_i(x) - \psi_i(y)} < \eps$
for all $i$.
This means that the family of maps $\{\psi_i\}_{i=1}^\infty$ is
equicontinuous.
Clearly, these maps are commonly bounded (by 0 from below and by 1
from above). Therefore, by the Arzela-Ascoli Theorem, there is a
subsequence of the sequence $(\psi_i)_{i=1}^\infty$, convergent to
some continuous map $\psi$.

The map $\psi$ is increasing and it maps $I$ onto $[0,1]$.
If $f$ is continuous and monotone on some $\chull{x,y}$ then
$f_{n_i}$ is continuous and monotone on $\chull{x,y}$.
Consequently, again by~\eqref{eq4a.1},
$
\abs{\psi_i(f_{n_i}(x)) - \psi_i(f_{n_i}(y))} =
  \beta_i\abs{\psi_i(x) - \psi_i(y)}
$
for all $i$. By passing to the limit we get
$\abs{\psi(f(x)) - \psi(f(y))} = \gamma \abs{\psi(x) - \psi(y)}$.
Therefore $\psi$ is $f$-compatible and,
by Lemma~\ref{E4a.2}, it semiconjugates $f$ to some
$g\in \CS_\gamma$. We get $\log \beta = h(f) \ge h(g) = \log \gamma
\ge h(f)$, and thus $\gamma = \beta$.
\end{proof}

\section{Transitive maps}\label{sec-trans}

It is known that for transitive continuous interval maps the
semiconjugacies to maps of constant slope are actually conjugacies.
The same is true for piecewise continuous maps.

\begin{proposition}\label{E4a.6}
If $f\in\NSM$ is transitive and $\psi\in \NDC$ semiconjugates it with
a map $\map{g}{[0,1]},$ then $\psi$ is a homeomorphism (so it
conjugates $f$ with $g$).
\end{proposition}

\begin{proof}
Suppose that $\psi$ is not a homeomorphism. Then there exists $z \in
[0,1]$ such that $\psi^{-1}(z)$ is a proper interval. By the
transitivity of $f$ there exists $x\in I$ with a dense trajectory. In
particular, this trajectory is dense in the interior of
$\psi^{-1}(z)$. Hence, there exist $n\ge 0$ and $k > 0$ such that
$f^n(x), f^{n+k}(x) \in\psi^{-1}(z)$. We have $g^k(z) =
g^k\circ\psi(f^n(x)) = \psi\circ f^k(f^n(x)) = \psi(f^{k+n}(x)) = z$,
so $z$ is a periodic point of $g$. However, since the $f$-trajectory
of $x$ is dense in $I$, the $g$-trajectory of $\psi(x)$ is dense in
$[0,1]$. On the other hand, the $g$-trajectory of $\psi(x)$
is finite, a contradiction.
\end{proof}

By Theorem~\ref{E4a.5} and Proposition~\ref{E4a.6}, we get the
following corollary.

\begin{MainCorollary}\label{E4a.7}
If $f\in \NSM$ is transitive and $h(f) = \log \beta > 0$ then $f$ is
conjugate to some map $g \in \CS_{\beta}$ via an increasing homeomorphism.
\end{MainCorollary}

\section{Operator}\label{sec-oper}

Let $\tcal$ be the space of all continuous transitive maps from
$\NSM$, and let $\tcal_n$ be its subspace consisting of all maps of
modality $n$ (that is, with $n+1$ laps).
We consider those spaces with the metric~\eqref{dist}.

By Corollary~\ref{E4a.7}, each element of $\tcal$ is conjugate via a
homeomorphism from $\NDC$ to a map of a constant slope and this slope
is the exponential of the topological entropy of our map. We will
prove that this map is unique.

\begin{lemma}\label{measures}
Assume that $f\in\tcal$ has constant slope. Then it has a unique
measure with maximal entropy.
This measure is ergodic and equivalent to the Lebesgue measure.
\end{lemma}

\begin{proof}
By~\cite{MS}, if the slope of $f$ is $\beta$, then the entropy of $f$
is $0$ if $\beta\le 1$ and $\log\beta$ if $\beta>1$. However, since
$f$ is transitive, by~\cite{Blokh}, it has positive entropy.
Therefore, $\beta>1$ and $h(f)=\log\beta$. By \cite[Theorem~4]{Hof},
$f$ has a unique measure $\mu$ of maximal entropy, and $\mu$ is
positive on every non-empty open set.
By uniqueness, $\mu$ is ergodic.

Now we use the results of~\cite{DKU}. The Lebesgue measure is
$\beta$-conformal. The Rokhlin Formula becomes
\[
\log\beta=\int\log\beta\; d\mu,
\]
so it is satisfied, since $\mu$ is a probability measure. Therefore by
Theorem~1 and Proposition~1 of~\cite{DKU}, $\mu$ is absolutely
continuous with respect to the Lebesgue measure.

Since $\mu$ is ergodic and absolutely continuous with respect to the
Lebesgue measure, by~\cite{LY}, it is equivalent to the Lebesgue
measure restricted to a finite union of intervals. However, since
$\mu$ is positive on every non-empty open set, we see that it
is equivalent to the Lebesgue measure on the whole interval $I$.
\end{proof}

\begin{theorem}\label{unique}
Let $f,g\in\tcal$ have constant slopes and assume that $f$ and $g$ are
conjugate via $\phi\in\NDC$.
Then $\phi$ is the identity.
\end{theorem}

\begin{proof}
Note first that since $f$ and $g$ are conjugate, we have $h(f) = h(g)$ and, by
Remark~\ref{entconstslope}, $f,g \in \CS_{\beta}$ for some $\beta>1$.
By Lemma~\ref{measures}, each of them has a unique measure of maximal
entropy (call them $\mu$ and $\nu$ respectively) and those
measures are ergodic and equivalent to the Lebesgue measure.
Then the map induced by $\phi$ in the space of probability measures
maps $\mu$ to $\nu$.
This shows that $\phi$ is absolutely continuous.
In particular, it is differentiable almost everywhere.
{}From $\phi \circ f = g \circ \phi$, it follows that
\[
\abs{f'(x)}\cdot\phi'(f(x)) = \phi'(x)\cdot\abs{g'(\phi(x))}
\]
for almost every $x\in I.$
Since $\abs{f'}$ and $\abs{g'}$ are constant and equal almost
everywhere, we see that the function $\phi'$ is $f$-invariant
(i.e. $\phi' \circ f = \phi'$) almost everywhere.
Consequently, since $\mu$ is ergodic, $\phi'$ takes a constant
value $s$ $\mu$-almost everywhere.
If $s<1$, then $\phi(I)$ is shorter than $I$;
if $s>1$ then $\phi(I)$ is longer than $I$; a contradiction in both
cases. Thus, $\phi'=1$ $\mu$-almost everywhere, so $\phi$ is the identity.
\end{proof}

\begin{corollary}\label{un}
A map $f\in\tcal$ of topological entropy $\log \beta$ is conjugate via an increasing
homeomorphism to exactly one map $g$ of constant slope $\beta$.
\end{corollary}

In such a way for each $n\ge 2$ we get an operator $\Phi_n$ from
$\tcal_n$ to the space of maps of constant slope of modality $n$. We
will investigate the continuity of this operator.

\begin{lemma}\label{diffslope}
Assume that $f,g\in\ccal$ are piecewise monotone maps with constant
slopes $\alpha$ and $\beta$ respectively, where $\alpha > \beta$ and
that the modality of $f$ is $n$.
Then $\dist(f,g) \ge \tfrac{\alpha-\beta}{2n+2}$.
\end{lemma}

\begin{proof}
Let $[a,b]$ be the longest lap of $f$. Since $f$ has $n+1$ laps, we
have $b-a\ge \tfrac{1}{n+1}$. We have $\abs{f(b)-f(a)} = \alpha(b-a)$,
while, since $g$ is Lipschitz continuous with constant $\beta$, we
have $\abs{g(b)-g(a)} \le \beta(b-a)$.
Therefore,
\begin{align*}
\abs{g(b)-f(b)} + \abs{g(a)-f(a)}
  &\ge \abs{g(b)-f(b)-g(a)+f(a)}\\
  &\ge \abs{f(a)-f(b)}-\abs{g(b)-g(a)}\\
  &\ge\alpha(b-a)-\beta(b-a).
\end{align*}
Thus, at least one of the numbers $\abs{g(b)-f(b)}$ and
$\abs{g(a)-f(a)}$ is larger than or equal to
$\tfrac{1}{2}(\alpha-\beta)(b-a)$.
Hence,
\[
\dist(f,g) \ge
  \frac{1}{2}(\alpha-\beta)(b-a) \ge
  \frac{\alpha-\beta}{2n+2}.
\]
\end{proof}

\begin{corollary}
If $f\in\tcal_n$ is a point of discontinuity of the topological
entropy as a function from $\tcal_n$ to $\R$, then $f$ is also a
point of discontinuity of the operator $\Phi_n$.
\end{corollary}

We conjecture that the operator $\Phi_n$ is continuous at every point
of continuity of the topological entropy on $\tcal_n$.

It is proven in~\cite{M1} that the topological entropy as a function
from $\tcal_n$ to $\R$ is discontinuous for every $n\ge 5$. Therefore
we get the next corollary.

\begin{corollary}
The operator $\Phi_n$ is discontinuous for every $n\ge 5$.
\end{corollary}

\section{Graph maps}\label{sec-graph}

Taking into account the considerations from Section~\ref{sec-pcont},
together with Theorem~\ref{E4a.5} (Section~\ref{sec-main})
and Corollary~\ref{E4a.7} (Section~\ref{sec-trans}),
we get easily the following result:

\begin{MainTheorem}\label{graphs}
If $\map{f}{G}$ is a continuous piecewise monotone graph map and
$h(f) = \log\beta > 0$, then $f$ is semiconjugate via a monotone map
to some continuous piecewise monotone graph map with constant slope.
Moreover, if $f$ is transitive, then this semiconjugacy is actually a
conjugacy.
\end{MainTheorem}

\begin{proof}
We partition the graph $G$ into pieces and align those pieces to get an interval
$I,$ as explained in Section~\ref{sec-pcont}.
This induces a piecewise continuous piecewise monotone map {\map{\widetilde{f}}{I}}
with the same entropy as $f$.
By Theorem~\ref{E4a.5}, $\widetilde{f}$ is semiconjugate to a map
{\map{\widetilde{g}}{I}} via some
$\widetilde{\psi} \in \NDC_{\widetilde{f}}.$

We know that $\widetilde{\psi}$ is monotone on each lap of $\widetilde{f}.$
That is, from the topological point of view, it just contracts some intervals to points.
Thus, its lifting $\psi$ to the level of graphs does the same.
In particular, it follows that $\psi$ is globally continuous and monotone.

In such a way, we get a semiconjugacy
$\psi$ of $f$ to some map $g$ of some graph $G'$ to itself.
Note that $G'$ is not necessarily equal to $G$ since $\psi$ may contract some
edges of $G$ to points.

To prove that $g$ is continuous, we only need to show that if $J$ is a
connected set which is contracted by $\psi$ to a point, then $f(J)$
is also contracted by $\psi$ to a point. This follows from the fact that
$\widetilde{\psi}$ is a semiconjugacy and the definition of the class
$\NDC_{\widetilde{f}}$.

If $f$ is transitive then instead of Theorem~\ref{E4a.5} we use
Corollary~\ref{E4a.7}.
This completes the proof.
\end{proof}

Let us stress that, as we pointed out in the proof,
the semiconjugacy can change the graph by collapsing some connected sets.

%Of course
We could also consider piecewise continuous piecewise
monotone graph maps, but they would be practically indistinguishable
from piecewise continuous piecewise monotone interval maps.

\bibliographystyle{plain}

\end{document}